\documentclass[11pt]{amsart}
\usepackage[utf8x]{inputenc}

\usepackage{amsmath,amssymb,amscd,amsthm}

\usepackage{url}
\usepackage[dvips]{graphicx}
\usepackage{epsfig}
\usepackage{graphics}
\usepackage{color}

\newtheorem{dfn}{Definition}[section]
\newtheorem{thm}[dfn]{Theorem}

\newtheorem{cor}[dfn]{Corollary}

\theoremstyle{remark}
\newtheorem{rem}[dfn]{Remark}

\newcommand{\Met}{{\mathit{Met}}}
\newcommand{\scal}{\operatorname{scal}}
\newcommand{\Ric}{\operatorname{Ric}}
\newcommand{\vol}{\operatorname{vol}}
\newcommand{\dvol}{\operatorname{dvol}}
\newcommand{\cE}{\mathcal{E}}
\newcommand{\R}{\mathbb R}
\newcommand{\const}{\mathrm{const}}
\newcommand{\Sph}{\mathbb{S}}
\renewcommand{\Vert}{\operatorname{Vert}}
\newcommand{\conv}{\operatorname{conv}}

\title[Variations of the discrete Hilbert-Einstein functional]{Variational properties of the discrete Hilbert-Einstein functional}
\author{Ivan Izmestiev}
\thanks{Supported by the European Research Council under the European Union's Seventh Framework Programme (FP7/2007-2013)/\allowbreak ERC Grant agreement no.~247029-SDModels}

\begin{document}

\begin{abstract}
This is a survey on rigidity and geometrization results obtained with the help of the discrete Hilbert-Einstein functional, written for the proceedings of the ``Discrete Curvature'' colloquium in Luminy.

\end{abstract}

\maketitle

\section{Introducing the functional}
\subsection{Smooth case}
Let $M$ be a smooth compact manifold without boundary.
In Riemannian geometry, the \emph{Hilbert-Einstein functional} is a function on the space $\Met_M$ of Riemannian metrics on $M$ which associates to a metric $g$ the integral of half its scalar curvature:
$$
S \colon \Met_M \to \R, \quad S(g) = \frac12 \int_M \scal_g \dvol_g
$$

If $\dim M = 2$, then we have $\scal_g = 2K_g$, where $K_g$ is the Gauss curvature. Hence by the Gauss-Bonnet theorem
$$
S(g) = 2\pi \chi(M)
$$
is independent of the metric $g$.

Starting from $\dim M = 3$, the functional $S$ becomes more interesting. Denote
$$
S'_h = \left.\frac{d}{dt}\right|_{t=0} S(g + th),
$$
where $h$ is the field of symmetric bilinear forms on $M$.
\begin{thm}
The first variation of $S$ is given by the formula
$$
S'_h = \frac12 \int_M \left\langle \frac{\scal_g}{2} g - \Ric_g, h \right\rangle \dvol_g
$$
\end{thm}

\begin{cor}
Let $\dim M \ge 3$.\nopagebreak
\begin{enumerate}
\item[a)]
A metric $g \in \Met_M$ is a critical point of $S$ if and only if $g$ is Ricci-flat, i.~e. $\Ric_g = 0$.
\item[b)]
Critical points of the restriction of $S$ to the space $\Met_M^1$ of metrics of unit total volume are Einstein metrics, i.~e. metrics with $\Ric_g = \lambda g$.
\end{enumerate}
\end{cor}

If $\dim M = 3$, then Einstein metrics are metrics of constant sectional curvature (Euclidean, hyperbolic or spherical).

See \cite[Chapter 4C]{Bes87} for details.

\subsection{Discrete case}
\label{sec:DiscrDef}
Let $M$ be a compact $3$-manifold without boundary.
Fix a triangulation $T$ of $M$ and pick a map
$$
\ell \colon \cE(T) \to (0, +\infty), \quad e \mapsto \ell_e
$$
assigning to every edge $e$ of $T$ a length $\ell_e$. Consider only those $\ell$ for which every tetrahedron of $T$ can be realized as a Euclidean tetrahedron with the edge lengths $\ell$. (This set is non-empty, since $\ell_e = 1$ for all $e$ will do.)

The map $\ell$ introduces a Euclidean metric on each tetrahedron of $T$, and a \emph{Euclidean cone-metric} on $M$. Note that different pairs $(T,\ell)$ can define the same metric; for example, we may subdivide the triangulation $T$ and define lengths of new edges appropiately.

The \emph{Hilbert-Einstein functional} on the space of Euclidean cone-metrics is
$$
S(T,\ell) = \sum_{e \in \cE(T)} \ell_e (2\pi - \omega_i),
$$
where $\omega_e$ is the total angle around $e$, see Figure \ref{fig:ConeAngle}. Clearly, the value of $S$ depends only on the metric, not on the choice of the representative $(T,\ell)$.

\begin{figure}[ht]
\centering
\begin{picture}(0,0)%
\includegraphics{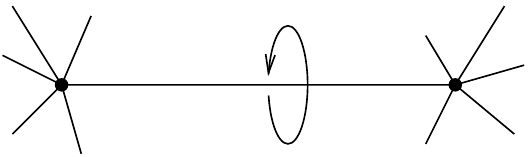}%
\end{picture}%
\setlength{\unitlength}{4144sp}%
\begingroup\makeatletter\ifx\SetFigFont\undefined%
\gdef\SetFigFont#1#2#3#4#5{%
  \reset@font\fontsize{#1}{#2pt}%
  \fontfamily{#3}\fontseries{#4}\fontshape{#5}%
  \selectfont}%
\fi\endgroup%
\begin{picture}(2409,699)(1069,-838)
\put(2487,-316){\makebox(0,0)[lb]{\smash{{\SetFigFont{12}{14.4}{\rmdefault}{\mddefault}{\updefault}{\color[rgb]{0,0,0}$\omega_e$}%
}}}}
\put(1781,-694){\makebox(0,0)[lb]{\smash{{\SetFigFont{12}{14.4}{\rmdefault}{\mddefault}{\updefault}{\color[rgb]{0,0,0}$\ell_e$}%
}}}}
\end{picture}%

\caption{Lengths and angles in a $3$-dimensional cone-manifold.}
\label{fig:ConeAngle}
\end{figure}
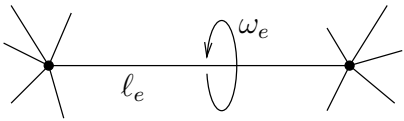

\begin{rem}
\label{rem:Approx}
If $\dim M = n$, then Euclidean cone-metrics on $M$ have cone singularities around codimension $2$ faces of $T$, and one puts
$$
S(T,\ell) = c_n \sum_{\dim F = n-2} \vol_{n-2}(F) (2\pi - \omega_F)
$$
Cheeger, M\"uller, and Schrader \cite{CMS84} have shown that the discrete Hilbert-Einstein functional converges to the smooth one if a sequence $(T^{(n)}, \ell^{(n)})$ of Euclidean cone-metrics converges to a Riemannian metric $g$ (with respect to the Lipschitz distance between metric spaces) so that all simplices in $(T^{(n)},\ell^{(n)})$ stay sufficiently fat.

It is an open problem to what most general class of metric spaces (including Riemannian manifolds and Euclidean cone-manifolds) the Hilbert-Einstein functional, and more generally, all total Lipschitz-Killing curvatures can be extended.
\end{rem}

The Hilbert-Einstein functional can also be defined for hyperbolic and spherical cone-metrics. In this case an additional volume term appears, see \cite[Sections 4.2-4.4]{Izm13+a}.

\subsection{Critical points in the discrete case}
Call the quantity $\kappa_e := 2\pi - \omega_e$ the \emph{curvature} of a Euclidean cone-metric at the edge $e$. Then we have $S(T,\ell) = \sum_e \ell_e \kappa_e$.

\begin{thm}
\label{thm:DiscrCritPt}
We have $\frac{\partial S}{\partial \ell_e} = \kappa_e$
\end{thm}
This is equivalent to the identity $\sum_e \ell_e d\kappa_e = 0$, which follows by adding up the Schl\"afli formula for all tetrahedra of $T$. An independent proof was given by the physicist Tullio Regge who introduced the discrete Hilbert-Einstein functional in \cite{Reg61}. In particular, Regge's argument provides an elementary proof of the Schl\"afli formula.

\begin{cor}
\label{cor:CritPts}
Critical points of the discrete Hilbert-Einstein functional represent flat metrics.
\end{cor}

Similarly, critical points of the functional on the space of hyperbolic cone-metrics (see end of Section \ref{sec:DiscrDef}) correspond to hyperbolic metrics without cone singularities.

Corollary \ref{cor:CritPts} has two applications:
\begin{itemize}
\item
Construct a metric of constant curvature by finding a critical point of $S$ .
\item
Prove rigidity of a space-form by showing non-degeneracy of the corresponding critical point of $S$.
\end{itemize}

It is surely tempting to try to reprove hyperbolization theorem by showing the existence of critical points of $S$ under suitable topological assumptions on $M$. Two main difficulties arise here. One is that the functional is neither convex nor concave, which makes existence of a critical point difficult to prove. The other is the choice of a triangulation $T$, since we cannot know in advance the combinatorial type of a geodesic triangulation. One possible solution is to start with an arbitrary triangulation and change its combinatorial type while deforming the metric. This is what was done in our proof of the Alexandrov theorem (Section \ref{sec:AlexThm}) which can be viewed as a simple case of geometrization with boundary conditions. See also \cite{FI09,FI11} for hyperbolic metrics on cusps with boundary.

In the smooth case, Blaschke and Herglotz \cite{BH37} suggested to use the variational property of the smooth Hilbert-Einstein functional for solving Weyl's problem. Yamabe's work \cite{Yam60} resulted from an attempt to solve Poincar\'e's conjecture using the same variational principle. Most recently, a geometrization program developing Yamabe's ideas was proposed by M.~Anderson \cite{And02,And03}.

The second of the above points, the infinitesimal rigidity, is more easily tractable. Variational properties of the Hilbert-Einstein functional form the basis of Koiso's proof of the infinitesimal rigidity of Einstein manifolds under certain assumptions on the eigenvalues of the curvature tensor, \cite{Koi78}. We used similar ideas in a new proof of the infinitesimal rigidity of convex polyhedra (Section \ref{sec:InfRigConv}) and of a class of non-convex polyhedra (Section \ref{sec:InfRigNon}).

\section{Infinitesimal rigidity of convex polyhedra}
\label{sec:InfRigConv}
\subsection{The boundary term of the Hilbert-Einstein functional}
If the compact manifold $M$ has a non-empty boundary, then the Hilbert-Einstein functional needs a boundary term, in order to remain differentiable. In the smooth case, this is
$$
S(g) = \frac12 \int_M \scal_g \dvol_g + \int_{\partial M} H \dvol^\partial_g,
$$
where $H$ is the trace of the second fundamental form $II$. The variational formula becomes
$$
S'_h = \frac12 \int_M \left\langle \frac{\scal_g}{2} g - \Ric_g, h \right\rangle \dvol_g + \frac12 \int_{\partial M} \left\langle Hg - II, h \right\rangle \dvol^\partial_g
$$

In the discrete case we have
$$
S(T,\ell) = \sum_{e \in \cE_i(T)} \ell_e (2\pi - \omega_i) + \sum_{e \in \cE_\partial(T)} \ell_e (\pi - \theta_e),
$$
where $\cE_i(T)$ and $\cE_\partial(T)$ are the sets of interior and boundary edges of $T$, respectively, and $\theta_e$ is the dihedral angle at the boundary edge $e$. The variational formula is obtained again by adding up the Schl\"afli formulas for individual simplices:
$$
\frac{\partial S}{\partial \ell_e} =
\begin{cases}
2\pi - \omega_e, &\text{if }e \in \cE_i(T)\\
\pi - \theta_e, &\text{if }e \in \cE_\partial(T)
\end{cases}
$$

\begin{rem}
If $M \subset \R^3$ is a convex body, then both of the above boundary terms appear as the coefficients at the $t^2$ term in the Steiner formula for $M$. Another common interpretation of both is $4\pi$ times the mean width (average length of projections to lines) of $M$. Check ball and cube.
\end{rem}

If we keep the metric on the boundary fixed (that is $h(X,Y) = 0$ for $X, Y \in T\partial M$ in the smooth case, and $\ell_e = \const$ for $e \in \cE_\partial(T)$ in the discrete case), then the critical points of the functionals are metrics that are flat inside $M$ and restrict to the given metric on the boundary.

\subsection{A proof of the infinitesimal rigidity of a convex polyhedron}
\label{sec:InfRigProof}
Let $P \subset \R^3$ be a compact convex polyhedron. For simplicity, assume that all faces of $P$ are triangles. An \emph{infinitesimal isometric deformation} of $P$ is an assignment of a vector $q_i$ to every vertex $p_i$ such that
$$
\langle p_i - p_j, q_i - q_j \rangle = 0 \quad \text{ for every edge }p_ip_j
$$
which is equivalent to $\left.\frac{\partial}{\partial t}\right|_{t=0} \|p_i(t) - p_j(t)\| = 0$ with $p_i(t) = p_i + tq_i$. A polyhedron is called \emph{infinitesimally rigid} if every its infinitesimal isometric deformation extends to an infinitesimal isometry of $\R^3$.

We will take another viewpoint: instead of deforming an embedded surface (the boundary of the polyhedron) we deform the metric inside the polyhedron itself. For this, choose a point $a$ inside $P$ and subdivide $P$ into triangular pyramids with apex $a$ and faces of $P$ as bases. This results in a triangulation $T$ of $P$. Denote by
$$
r_i := \|a - p_i\|, \quad \ell_{ij} := \|p_i - p_j\|
$$
the lengths of interior and boundary edges, respectively.
We will change $r_i$ while keeping $\ell_{ij}$ fixed and look what happens with the curvatures $\kappa_i$ around the interior edges (at the beginning we have $\kappa_i = 0$). See Figure \ref{fig:WarpPoly}.

\begin{figure}[ht]
\centering
\begin{picture}(0,0)%
\includegraphics{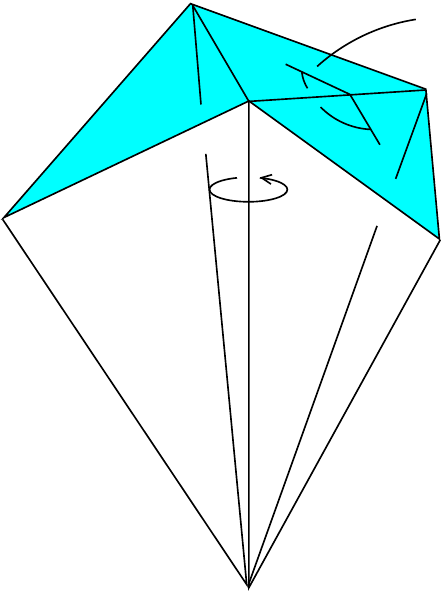}%
\end{picture}%
\setlength{\unitlength}{4144sp}%
\begingroup\makeatletter\ifx\SetFigFont\undefined%
\gdef\SetFigFont#1#2#3#4#5{%
  \reset@font\fontsize{#1}{#2pt}%
  \fontfamily{#3}\fontseries{#4}\fontshape{#5}%
  \selectfont}%
\fi\endgroup%
\begin{picture}(2025,2694)(-226,-1693)
\put(1745,815){\makebox(0,0)[lb]{\smash{{\SetFigFont{12}{14.4}{\rmdefault}{\mddefault}{\updefault}{\color[rgb]{0,0,0}$\theta_{ij}$}%
}}}}
\put(1044,-61){\makebox(0,0)[lb]{\smash{{\SetFigFont{12}{14.4}{\rmdefault}{\mddefault}{\updefault}{\color[rgb]{0,0,0}$\omega_i$}%
}}}}
\put(976,-627){\makebox(0,0)[lb]{\smash{{\SetFigFont{12}{14.4}{\rmdefault}{\mddefault}{\updefault}{\color[rgb]{0,0,0}$r_i$}%
}}}}
\end{picture}%
\caption{Lengths and angles in the triangulation $T$ of the polyhedron $P$. Shaded triangles lie on the boundary; only a part of the triangulation shown.}
\label{fig:WarpPoly}
\end{figure}
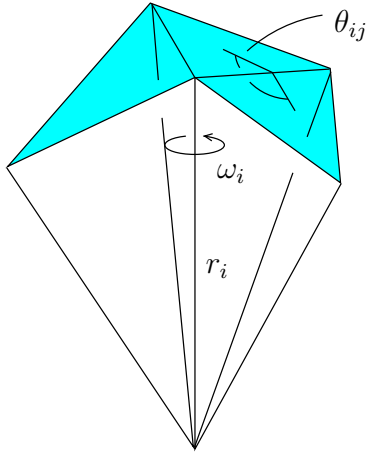

\begin{dfn}
A deformation $(s_i)$ of the interior edge lengths $(r_i)$ is called \emph{curvature-preserving}, if all directional derivatives
$$
\frac{d \kappa_i}{ds} = \sum_i \frac{\partial \kappa_i}{\partial r_j} s_j
$$
vanish. In other words, if
$$
s \in \ker \left( \frac{\partial \kappa_i}{\partial r_j} \right) = \ker \left( \frac{\partial^2 S}{\partial r_i \partial r_j} \right)
$$
\end{dfn}

Among curvature-preserving deformations there are trivial ones that result from a displacement of the point $a$ inside $P$. It is easy to show that they form a $3$-dimensional subspace. Also it is easy to see that the infinitesimal rigidity of $P$ in the original sense is equivalent to the absence of non-trivial curvature-preserving deformations:
$$
P \text{ is infinitesimally rigid } \Leftrightarrow \dim \ker \left( \frac{\partial^2 S}{\partial r_i \partial r_j} \right) = 3
$$

The following theorem implies that convex polyhedra are infinitesimally rigid.
\begin{thm}
Let $P$ be a compact convex polyhedron with triangular faces. 
Then the second variation $\left( \frac{\partial^2 S}{\partial r_i \partial r_j} \right)$ for the star-like triangulation of $P$ described above has the signature $(+, 0, 0, 0, -, \ldots, -)$.
\end{thm}
The part about the rank of the second variation is proved in \cite[Section 3]{Izm13+a}. The fact that the positive index is equal to 1 follows from the coincidence of the second variations of $S$ and of the volume of polar dual \cite[Section 4.1]{Izm13+a}, and from the signature of the second variation of the volume, provided by the second Minkowski inequality for mixed volumes, \cite[Appendix]{Izm10}.

\section{Alexandrov's theorem}
\label{sec:AlexThm}
Alexandrov's theorem \cite{Ale42} states the existence and uniqueness of a compact convex polyhedron in $\R^3$ with a prescribed boundary metric. The intrinsic boundary metric is a Euclidean cone-metric (since the surface of a polyhedron can be glued from triangles) with singular points of positive curvature (vertices of the polyhedron). Note that the intrinsic metric does not detect the edges of a polyhedron.

\begin{thm}[A.~D. Alexandrov, \cite{Ale42}]
Let $g$ be a Euclidean cone-metric on the sphere with singular points of positive curvature. Then there exists a unique up to congurence compact convex polyhedron in $\R^3$ with $g$ as the intrinsic metric on the boundary. (The polyhedron may also be a polygon, in which case instead of the intrinsic metric on the boundary two copies of the polygon glued along pairs of corresponding edges are taken.)
\end{thm}

In \cite{BI08} a new proof of Alexandrov's theorem was given, similar in the spirit to the proof of the infinitesimal rigidity described in Section \ref{sec:InfRigProof}.

We start with a certain geodesic triangulation $\bar T(0)$ of the sphere equipped with metric $g$, with vertices at the singular points, and an assignment of a positive number $r_i(0)$ to every singular point $p_i$. This allows us to construct a Euclidean cone-manifold $P(\bar T(0), r(0))$ by gluing together triangular pyramids with radial edge lengths $r_i(0)$ and triangles of $\bar T_0$ as bases. Namely, we take the Delaunay triangulation of $(\Sph^2, g)$ as $\bar T(0)$, and put $r_i(0) = R$ for all $i$, with $R$ sufficiently large. This ensures that pyramids exist and that the ``warped polyhedron'' $P(\bar T(0), r(0))$ is convex at the boundary (i.~e. $\theta_{ij}(0) \le \pi$).

Then we proceed by deforming the lengths $r_i$, thus obtaining a continuous family of warped polyhedra $P(\bar T(t), r(t))$. The deformation is chosen so that
\begin{itemize}
\item $\kappa_i(t) = (1-t) \kappa_i(0)$, where $\kappa_i(t)$ is the curvature at the edge $ap_i$ in $P(\bar T(t), r(t))$;
\item the dihedral angles on the boundary remain $\le \pi$.
\end{itemize}

The second condition requires that at certain moments $t_1 < t_2 < \ldots$ the triangulation $\bar T(t)$ must be changed. The triangulation is determined uniquely (up to ``flat edges'', those where the dihedral angle is $\pi$) since the second condition is equivalent to $\bar T(t)$ being the weighted Delaunay triangulation of $(\Sph^2, g)$ with weights $r_i^2$, see \cite[Section 2.5]{BI08}.

The existence of a deformation satisfying the first condition follows from the non-degeneracy of the matrix $\left( \frac{\partial \kappa_i}{\partial r_j} \right)$ under certain assumptions \cite[Theorem 3.11--Proposition 3.16]{BI08}. In the limit as $t \to 1$ we have $\kappa_i \to 0$ for all $i$, thus $P(\bar T(1), r(1))$ is a compact convex polyhedron with a given metric on the boundary.

A corresponding numerical algorithm was implemented in a computer program by Stefan Sechelmann \cite{Sech}, see Figure \ref{fig:Comp}.

\begin{figure}[ht]
\centering
\includegraphics[width=.8\textwidth]{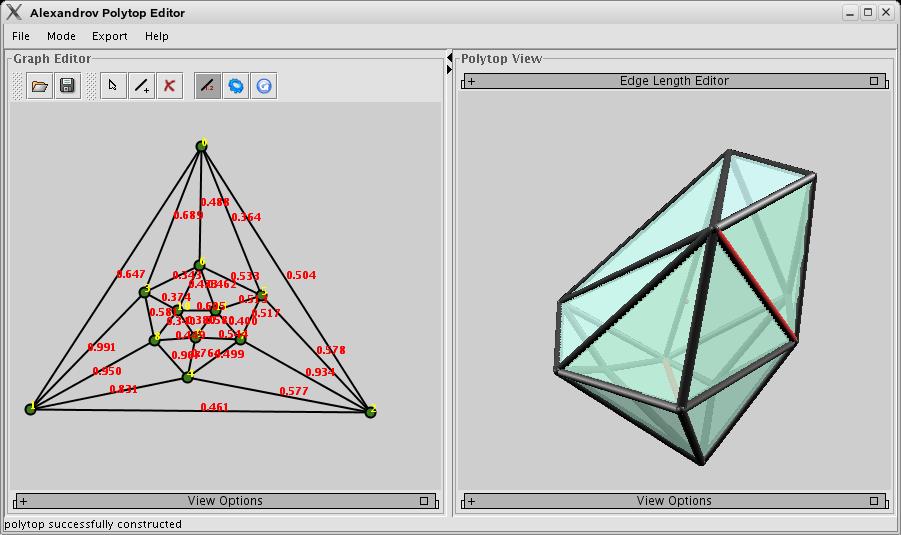}
\caption{A screenshot of \cite{Sech}.}
\label{fig:Comp}
\end{figure}

\section{Infinitesimal rigidity of weakly convex (co)decomposable polyhedra}
\label{sec:InfRigNon}
Infinitesimally flexible non-convex polyhedra exist, see Figure \ref{fig:SchJess}.
\begin{figure}[ht]
\centering
\includegraphics[height=.2\textheight]{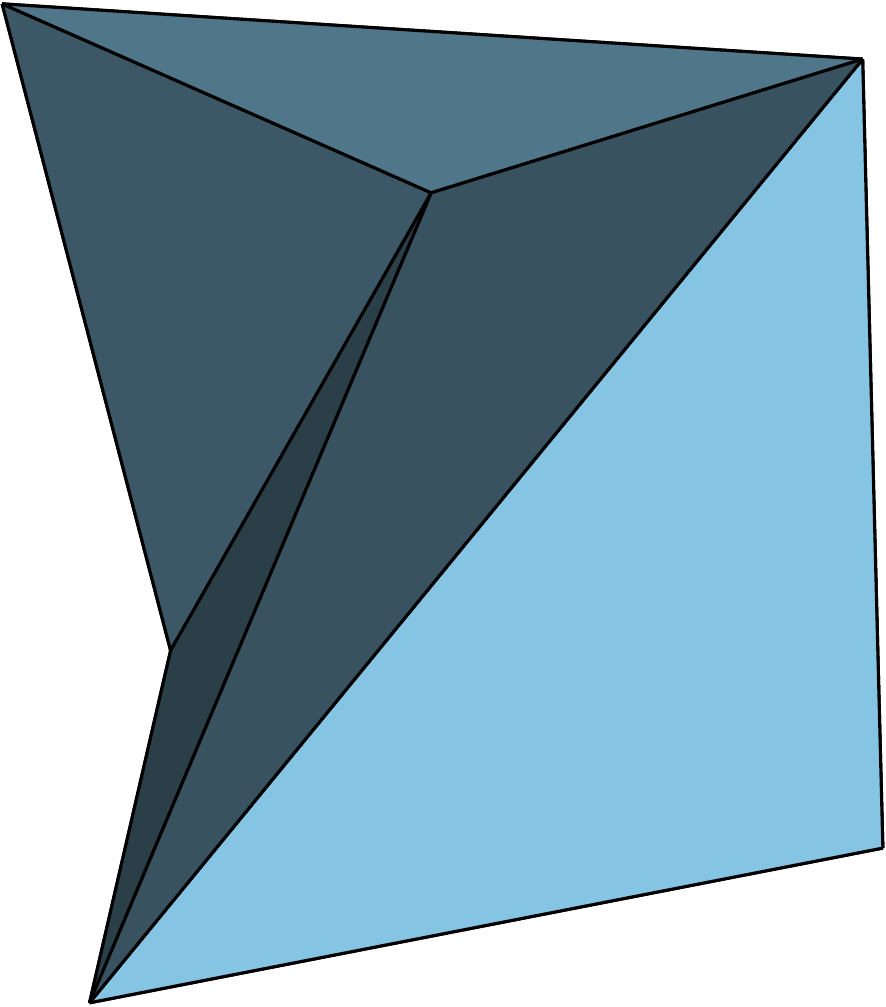} \hspace{.1\textwidth} \includegraphics[height=.2\textheight]{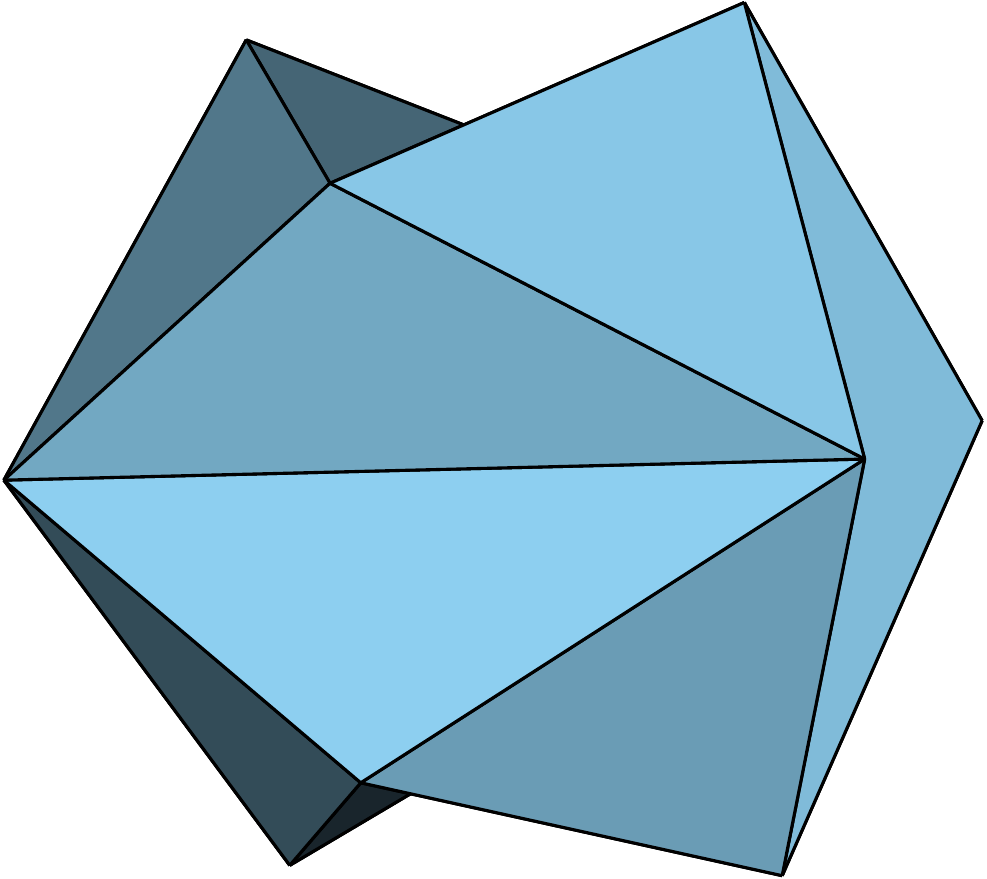}
\caption{Schoenhardt's twisted octahedron and Jessen's orthogonal icosahedron are infinitesimally flexible.}
\label{fig:SchJess}
\end{figure}

In \cite{IS10} the infinitesimal rigidity was proved for a wide class of non-convex polyhedra.

\begin{dfn}
A non-convex polyhedron is called \emph{weakly convex}, if its vertices lie in a convex position: $\Vert(P) = \Vert(\conv P)$.

A weakly convex polyhedron $P$ is called \emph{decomposable} if it can be triangulated without adding new vertices. It is called \emph{decomposable and codecomposable} if there is a triangulation $T$ of $\conv P$ such that $\Vert(T) = \Vert(P)$ and $P$ is a subcomplex of $T$.
\end{dfn}

Both polyhedra on Figure \ref{fig:SchJess} are weakly convex but not decomposable.

\begin{thm}[\cite{IS10}, Theorem 1.7]
\label{thm:WeakConv}
Weakly convex decomposable and codecomposable polyhedra are infinitesimally rigid.
\end{thm}

This theorem is a consequence of the following property of the Hilbert-Einstein functional.

\begin{thm}[\cite{IS10}, Theorem 1.17]
\label{thm:SignPoly}
Let $T$ be a triangulation of a convex polyhedron. Denote by $i$ the number of vertices of $T$ in the interior of the polyhedron and by $b$ the number of vertices in the interiors of its faces. (The number of vertices on the edges is irrelevant.)

Consider Euclidean cone-metrics inside the polyhedron arising from deformations of the interior edges of the triangulation. Then the matrix
\begin{equation}
\label{eqn:JacHess}
\left(\frac{\partial \kappa_e}{\partial \ell_f}\right) = \left(\frac{\partial^2 S}{\partial \ell_e \partial \ell_f}\right)
\end{equation}
has corank $3i + b$ and exactly $i$ positive eigenvalues.
\end{thm}

\begin{cor}
Let $T$ be a triangulation of a convex polyhedron that uses only vertices of this polyhedron. Then the matrix \eqref{eqn:JacHess} is negative definite.
\end{cor}

\begin{proof}[Proof of Theorem \ref{thm:WeakConv}]
In the triangulation $T$ of $\conv P$, take the subcomplex $\bar T$ that triangulates $P$. The Hessian matrix of $S$ for $\bar T$ is a principal minor of the Hessian of $S$ for $T$. Since the latter is negative definite, so is the former. In particular, it is non-degenerate. Hence it is impossible to change the lengths of interior edges of $\bar T$ without changing the curvatures in the first order. Thus $P$ is infinitesimally rigid.
\end{proof}

\begin{rem}
In the smooth case, the space of all infinitesimal deformations of a Riemannian metric can be decomposed as a direct sum of conformal, trivial, and anti-conformal deformations. The restriction of the second variation $D^2S$ to the space of conformal deformations is positive definite; trivial deformations don't change the value of $S$; and on the space of the anti-conformal deformations $D^2S$ is negative definite, provided that the spectrum of the curvature operator satisfies certain assumptions, \cite[Chapters 4G, 12H]{Bes87}.

In the discrete case, trivial deformations arise from arbitrary displacements of the interior vertices and from displacements of vertices inside the faces orthogonally to those faces. This space has dimension $3i+b$. Conformal deformations should correspond in ``blowing up'' at each vertex independently, thus their space has dimension $i$. Thus the signature of $D^2S$ as stated in Theorem \ref{thm:SignPoly} fits very well with what is known in the smooth case.
\end{rem}

Among other works dealing with the signature of the second variation of the discrete Hilbert-Einstein functional let us mention \cite{CR96, Gli11}.

\section{Directions for the future research}
Let $M$ be a closed hyperbolic $3$-manifold with a geodesic triangulation $T$. Then the infinitesimal rigidity of $M$ (also known as Calabi-Weil rigidity \cite{Wei60,Cal61}) is equivalent to $\dim\ker D^2S = 3i$, where $i$ is the number of vertices of $T$. It should be possible to determine the rank of $D^2S$ (or even better, the signature) by a sort of discrete Bochner method. This would yield a new proof of the Calabi-Weil theorem.
A similar method should work for cone-manifolds.
If $M$ is a hyperbolic cone-manifold, then $M$ is infinitesimally rigid (in the sense that any deformation preserving the cone-angles is trivial) provided that $\omega_e \le 2\pi$ around all edges $e$, \cite{MM11,Wei13}; without this assumption $M$ may be infinitesimally flexible, \cite{Izm11}.

For ideal triangulations of hyperbolic manifolds, the functional is concave. This makes cusped manifolds the first case to try to reprove the hyperbolization theorem.

The functional is concave also for semiideal triangulations, if all finite vertices lie on the boundary. This was used in \cite{FI09} to prove the existence of a hyperbolic cusp with a given cone-metric on the boundary. A generalization of this would be realizability of an arbitrary metric with curvature bounded below by $-1$ on the boundary of some hyperbolic cusp. On one hand, this should follow from the polyhedral case by an approximation argument; on the other hand, it would be interesting to find a variational proof that uses an extension of the Hilbert-Einstein functional to more general metric spaces. In particular, this is related to the problem at the end of Remark \ref{rem:Approx}.

\end{document}